\documentclass[11pt,eqno]{amsart}
\usepackage{amsmath,amssymb,txfonts}
\usepackage{amssymb}
\usepackage{amsxtra}
\usepackage{amsmath}
\usepackage{txfonts}
\usepackage[greek,english]{babel}
 \usepackage{color}
\usepackage{slashed}

\usepackage[cp1252]{inputenc}

\usepackage{mathrsfs}
\usepackage{graphicx}

 \usepackage[active]{srcltx}

\allowdisplaybreaks

\def\rr{{\mathbb R}}

\def\fz{\infty}

\def\ez{\varepsilon}

\def\bz{\beta}

\def\Xint#1{\mathchoice	{\XXint\displaystyle\textstyle{#1}}	{\XXint\textstyle\scriptstyle{#1}}	{\XXint\scriptstyle\scriptscriptstyle{#1}}	{\XXint\scriptscriptstyle\scriptscriptstyle{#1}}	\!\int}\def\XXint#1#2#3{{\setbox0=\hbox{$#1{#2#3}{\int}$}	\vcenter{\hbox{$#2#3$}}\kern-.5\wd0}}\def\dashint{\Xint-}

\newtheorem{thm}{Theorem}[section]
\newtheorem{lem}{Lemma}[section]

\newtheorem{defn}{Definition}[section]

\numberwithin{equation}{section}

\begin{document}



\arraycolsep=1pt

\title{Second order Sobolev regularity for normalized parabolic 
$p(x)$-Laplace equations via the algebraic structure }

\author{Yuqing Wang and Yizhe Zhu}
\address{ Department of Mathematics, Beihang University, Beijing 100191, P.R. China}
                    \email{ yuqingwang@buaa.edu.cn\\yizhezhu@buaa.edu.cn}

\thanks{     }

\date{\today }
\maketitle

\begin{center}
\begin{minipage}{13.5cm}\small{\noindent{\bf Abstract}\quad
Denote by $\Delta$ the Laplacian and by $\Delta_\infty$ the $\infty$-Laplacian. A fundamental inequality is proved for the algebraic structure of $\Delta v\Delta_\infty v$: for every $v\in C^{\infty}$,
$$\bigg| |D^2vDv|^2-\Delta v\Delta_\infty v-\frac{1}{2}[|D^2v|^2-(\Delta v)^2]|Dv|^2\bigg|
\le\frac{n-2}{2}[|D^2v|^2|Dv|^2-|D^2vDv|^2]$$
Based on this, we prove the result: When $n\ge2$ and $p(x)\in(1,2)\cup(2,3+\frac{2}{n-2})$, the viscosity solutions to parabolic normalized $p(x)$-Laplace equation have the $W^{2,2}_{loc}$-regularity in the spatial variable and the $W^{1,2}_{loc}$-regularity in the time variable

}
{\bf key words} secong order regularity, $p(x)$-Laplacian, variable exponents, generalized Lebesgue and Sobolev spaces
\end{minipage}
\end{center}

\section{Introduction}
Denote by $\Delta$ and $\Delta_\infty$ the Laplacian and $\infty$-Laplacian, respective, in $\mathbb{R}^n$ with $n\ge2$, i.e.
$$\Delta v={\rm div}(Dv) \quad {\rm and}\quad \Delta_\infty=D^2vDv\cdot Dv$$
Recall that, the following identity in the plane
\begin{equation}\label{eq1.1}
|D^2vDv|^2-\Delta v\Delta_\infty v=\frac{1}{2}[|D^2v|^2-(\Delta v)^2]|Dv|^2
\end{equation}
is the key to the higher order Sobolev regularity of infinity harmonic function in the plane established in \cite{2019An}. In this paper, we generalize \eqref{eq1.1} to the higher dimension: For every $v\in C^{\infty}$
\begin{equation}\label{eq1.2}
\bigg| |D^2vDv|^2-\Delta v\Delta_\infty v-\frac{1}{2}[|D^2v|^2-(\Delta v)^2]|Dv|^2\bigg|
\le\frac{n-2}{2}[|D^2v|^2|Dv|^2-|D^2vDv|^2]
\end{equation}
see \cite{2019Second}.

It turns out that \eqref{eq1.2} is a basic tool to study the second order Sobolev regularity of normalized parabolic $p(x,t)$-Laplacian. Here, for $p(x,t)\in C^1(\Omega_T)$ 
and $1<p_{-}:={\rm inf}\ p(x,t)\le p(x,t)\le{\rm sup}\ p(x,t)=:p_{+}<\infty$. we consider the viscosity solution of the normalized $p(x,t)$-Laplace equation
\begin{equation}\label{eq1.3}
u_t-\Delta_{p(x,t)}^N u=0
\end{equation}
where $\Delta_{p(x,t)}^N$ is the normalized $p(x,t)$-Laplace operator defined as
$$\Delta_{p(x,t)}^N u:=\Delta u- (p(x,t)-2)\frac{\Delta_\infty u}{|Du|^2}$$
Throughout the whole paper, $\Omega$ is always a domain of $\mathbb{R}^n$
\begin{thm}
Let $n\ge2$, $f\in W^{1,\fz}(\Omega)$, $p(x,t)\in C^1(\Omega_T)$ and $p(x,t)\in(1,3+\frac{2}{n-2})$. Then for any vicosity solution $u=u(x,t)$ to
\begin{equation}\label{eq1.4}
u_t-\Delta_{p(x)}^N u=f \quad {\rm in} \quad \Omega_T:=\Omega\times(0,T)
\end{equation}
we have $u_t$, $D^2u\in L^2_{loc}(\Omega_T)$.
Moreover, $D(|Du|^{\frac{p(x)-2}2}Du)\in L_{loc}^2(\Omega_T)$ and
\begin{align}\label{neq1.5}
\dashint_{\frac12Q}|D[|Dv|^{\frac{p(x)-2}2}Dv]|^2dx\,dt\le &C(n,p^B_\pm)\dashint_{Q}|Dv|^{p(x)-1}|Df|\,dx\,dt+C(n,p^B_\pm)(\frac1R+\frac1{R^2})\dashint_{Q}|Dv|^{p(x)}\,dx\,dt+\dashint_Q|Du|^{p(x)-1}\ln Du.\nonumber
\end{align}
\end{thm}

\large\textbf {1.1 $p(x)$-Laplace equation and its normalization}\\

To start with, we consider the $p(x)$-Laplace equation
$$\Delta_{p(x)}u:={\rm div}(|Du|^{p(x)-2}Du)=0$$

It's natural generalization of the $p$-Laplace operator, and  is considered first by Zhikov in \cite{1987Averaging}. 
The existence of the $p(x)$-Laplace equation is proved by Fan in \cite{2003Existence} when $p(x)\in C^0$. 
Petri, Teemu and Mikko proved the equivalence of viscosity and weak solutions for the $p(x)$-Laplace equation, Harnak's inequality and comparison principle in \cite{Petri2010Equivalence} when $p(x)\in C^1$. 
Moreover, Zhang proved the strong maximum principle for the equations with nonstandard $p(x)$-growth conditions in \cite{2005A}.

For the normalized $p(x)$-Laplace equation
$$\Delta_{p(x)}^Nu:=\Delta u+(p(x)-2)\frac{\Delta_\infty u}{|Du|^2}$$
There has been recent interest in normalized equations, see for example\cite{ISI000404699000004, 2015Modica}. We are partly motivated by the connection to stochastic tug-of-war games\cite{2011Tug} as the case of space dependent probabilities leads to\cite{2016Tug}.

There are some regularity results, when $p(x)$ is Lipschitz continuous, Siltakoski showed that viscosity solution is locally $C^{1,\alpha}$ in \cite{2018Equivalence}, by means of equivalence between viscosity
solutions and weak solutions to strong $p(x)$-Laplace equation
$$\Delta_{p(x)}^Su:=|Du|^{p(x)-2}\Delta_{p(x)}^Nu.$$ The local $ C^{1,\alpha} $ regularity of weak solutions of strong $p(x)$-Laplace equation has been obtained by Zhang-Zhou in \cite{2012Holder} when $p(x)\in\mathcal{P}^{log}$. For the approximation equation
$$-{\rm div}[(\kappa+|Du|^2)^{\frac{p(x)-2}{2}}Du]=f$$
It's known to have $C^{1,\alpha}_{loc}$ regularity (see\cite{2001Regularity,1999H} for $f=0$ and \cite{2007Global} for a
more general framework). Moreover, when $p(x)$ is Lipschitz continuous,  Challal and Lyaghfouri proved the $W^{2,2}_{loc}$ regularity when $\kappa>0$ and  $W^{2,2}_{loc}(p[\cdot\le2])$ when $\kappa=0$ in\cite{2011Second}

\large\textbf {1.2 Parabolic $p(x,t)$-Laplace equation and its normalization}

More recently, parabolic problems with variable exponents have been studied quite extensively, see for example \cite{2010Vanishing,2007On}.
The importance of investigating these problems lies in their occurrence in
modeling various physical problems involving strong anisotropic phenomena
related to electrorheological fluids\cite{2004Regularity}, image processing\cite{Li2010Variable} and also mathematical biology\cite{2010Positive}. 

We call
$$u_t-\Delta_{p(x,t)}u=0$$
parabolic $p(x,t)$-Laplace equation. It's natural generalization of the parabolic $p$-Laplace equation. In \cite{2016Quasilinear}, Giacomoni, Tiwari and Warnault discuss the existence, uniqueness and stabilization of weak solution provided that $p(x)\in\mathcal{P}^{log}$. For self improving property of higher integrability, Antontsev and Zhikov proved it when $p(x)\in\mathcal{P}^{loc}$ and $\frac{2n}{n+2}<p_{-}\le p(x,t)\le p_{+}<\infty$ in \cite{2005Higher}. For H\"{o}lder regularity, Shmarev proved that $u\in C^{1,\frac{1}{2}}(\Omega_T\cap \{t>0\})$ when $p(x,t)$ is H\"{o}lder continuous and $\frac{2n}{n+2}\le p(x,t)\le p_{+}<\infty$ in\cite{2018On}. Moreover, Mengyao, Chao and Shulin established global boundedness and H\"{o}lder regularity of the weak solution to a general $p(x,t)$-Laplace parabolic equation see\cite{2020Global}. For Calder\'{o}n-Zygmund estimate, Baroni and B\"{o}gelein established it in \cite{2014Calder} when $\frac{2n}{n+2}<p_{-}\le p(x,t)\le p_{+}<\infty$ and $p(x,t)\in\mathcal{P}^{log}$.

We call
$$u_t-\Delta_{p(x,t)}^Nu=0$$
normolized parabolic $p(x,t)$-Laplace equation. It arises naturally from a two-player zero-sum stochastic differential game (SDG) with probabilities depending on space and time, please see\cite{2016Local}. Parviainen-Ruosteenoja, in \cite{2016Local} proved the H\"{o}lder and Harnack estimates for a more general game that was called $p(x,t)$-game without using the PDE techniques and showed that the value functions of the game converge to the unique viscosity solution of the Dirichlet problem to the normalized $p(x,t)$-parabolic equation

$$(n+p(x,t))u_t(x,t)=\Delta_{p(x)}^Nu(x,t)$$
Moreover, Chao estabalished the interior H\"{o}lder regularity of the spatial gradient of viscosity solutions in \cite{2020Gradient}.

For $W^{2,2}$ regularity, Lindqvist \cite{2018Regularity} proved the $W^{2,2}_{loc}$ regularity in the spatial variables for viscosity sulotions to 
$$u_t-\Delta_pu=0$$
when $\frac{6}{5}<p<\frac{14}{5}$ and the $W^{1,2}_{loc}$ regularity in the time variable when $1<p<\frac{14}{5}$. Thier merhod is not capable to reach all $p\in(0,\infty)$. They also explained that, through the Cordes condition(see\cite{Maugeri2000Elliptic}), it is possible to get analogue results for $p$ in some restricted range but not all $p\in(1,\infty)$, mainly since the absence of zero boundary values produces many undesired
terms are hard to esitimates. Indeed, by the parabolic version of the Cordes condition, the only possible range to get the $W^{2,2}_{loc}$ regularity is $p\in(1,3+\frac{2}{n-1})$. Hongjie, Fa, Yi and Yuan in \cite{2019Second} improved the range of $p$ and get the $W^{2,2}_{loc}$ regularity in the spatial variables and $W^{1,2}_{loc}$ in the time variable when $1<p<3+\frac{2}{n-2}$. 

Additionally, the question was raised naturally:

{\em When $1<p(x,t)<3+\frac{2}{n-2}$, whether viscosity solutions to (1.3) enjoy the $W^{2,2}_{loc}$ regularity in spatial variables and $W^{1,2}_{loc}$ in the time variable}

\large\textbf {1.2 Ideas of the proof}\\
Theorem 1.1 is proved in section 3. Let $u=u(x,t)$ be a viscosity solution to\eqref{eq1.4}. Given any smooth domain
$U\subset\subset\Omega$, for $\epsilon\in(0,1]$ we let $u^\epsilon\in C^{1,\alpha}$ be a viscosity solution to the
regularized equation
\begin{equation}\label{eq1.5}
u^\epsilon_t-\Delta u^\epsilon-(p(x)-2)\frac{\Delta_\infty u^\epsilon}{|Du^\epsilon|^2+\epsilon}=0
\quad {\rm in} \quad U; \quad u=u^\epsilon\quad {\rm on} \ \partial U_T
\end{equation}
Applying\eqref{eq1.2} to $u^\epsilon$, one gets
\begin{align*}
&\frac{n}{2}(p(x)-2)^2\frac{|D^2u^{\epsilon}Du^{\epsilon}|^2}{|Du^{\epsilon}|^2+\epsilon}+[(p(x)-2)-\frac{n-2}{2}(p(x)-2)^2](\Delta u^{\epsilon})^2\\
&\quad\le\frac{n-1}{2}(p(x)-2)^2[|D^2u^{\epsilon}|^2-(\Delta u^{\epsilon})^2]\\
&\qquad+\frac{\epsilon}{2}(p(x)-2)^2\frac{(\Delta u^{\epsilon})^2-|D^2u^{\epsilon}|^2}{|Du^{\epsilon}|^2+\epsilon}
+(p(x)-2)\Delta u^{\epsilon}u_t^\epsilon \quad{\rm in} \ U_T
\end{align*}
Here we have the term
$$(p(x)-2)\Delta u^{\epsilon}u_t^\epsilon$$
from the parabolic structure, and also an annoying term
$$\frac{\epsilon}{2}(p(x)-2)^2\frac{(\Delta u^{\epsilon})^2-|D^2u^{\epsilon}|^2}{|Du^{\epsilon}|^2+\epsilon}$$
from the approximation procedure, either of which cannot be removed. With additional ideas and careful calculations, we bound such two terms when $p(x)$ just fluctuates within a certain range, if not we select a smooth function vanishing out the range.
(See section 3 by considering two cases via different methods)
\section{Preliminaries}
Because \eqref{eq1.3} is not in divergence form, the concept of weak solutions with test functions under the integral sign is problematic. Thus, in this section we first recall the definition of viscosity solution to \eqref{eq1.4}.
\begin{defn}
A lower (resp.upper) semicontinuous function u in $\Omega $ is a viscosity supersolution (resp.subsolution) to \eqref{eq1.4},
if for any $\phi\in C^2(\Omega)$, $u-\phi$ reaches the local minimum at $(x_0,t_0)\in\Omega$,
then when $D\phi(x_0,t_0)\neq0$, it holds that
$$\phi_t\ge(\le,resp.)\Delta\phi+(p(x)-2)\bigg<D^2\phi \frac{D\phi}{|D\phi|}, \frac{D\phi}{|D\phi|}\bigg>$$
at $(x_0,t_0)$; when $D\phi(x_0,t_0)=0$, it holds that
$$\phi_t\ge(\le,resp.)\Delta\phi+(p(x)-2)\big<D^2\phi q, q\big>$$
at $(x_0,t_0)$ for some $q\in \overline{B_1(0)}\subset\mathbb{R}^n$. A function u is a viscosity solution to \eqref{eq1.4} if and only if it is both viscosity supsolution and subsolution.
\end{defn}
Next, we state some known results about structures for $\Delta v\Delta_\infty v$ and $|D^2v|^2-(\Delta v)^2$,
which will play a center role in this paper. We begin with the following lemmas(see\cite{2019Second})
\begin{lem}
Let $n\ge2$ and $U$ be a domain on $\mathbb{R}^n$. For any $v\in C^{\infty}$, we have
\begin{align*}
\bigg| |D^2vDv|^2-\Delta v\Delta_\infty v-&\frac{1}{2}[|D^2v|^2-(\Delta v)^2]|Dv|^2\bigg|\\
&\le\frac{n-2}{2}[|D^2v|^2|Dv|^2-|D^2vDv|^2]\tag{2.1}
\end{align*}
\end{lem}

\begin{lem}
For any $v\in C^{\infty}$, $\phi\in C^\infty_c(U)$ and $c\in\mathbb{R}^n$, we have
\begin{align*}
\bigg| \int_{U}[|D^2v|^2-(\Delta v)^2]\phi^2dx\bigg|\le C\int_{U}|Dv-c|^2[|\phi||D^2\phi|+|D\phi|^2]\tag{2.2}
\end{align*}
\end{lem}

\begin{lem}\label{lem2.3}
Let   $\bz\in\rr$ and $\ez>0$.
For any   $v\in C^2(\Omega)$, any $\psi\in C_c^\fz(\Omega) $ and any vector $\vec{c}\in\rr^n$, we have

\begin{align}\label{id2.4}
&\int_\Omega\sigma_2(D[(|Dv|^2+\ez)^{{\bz/2}}Dv])\psi\,dx\\
&\quad =\int_\Omega\sum\limits_{1\leq i<j\leq n} (|Dv|^2+\ez)^{\bz/2 }v_{x_i}-c_i][(|Dv|^2+\ez)^{\bz/2 }v_{x_j}]_{x_j}\psi_{x_i}\,dx\nonumber\\
&\quad\quad -\int_\Omega\sum\limits_{1\leq i<j\leq n} (|Dv|^2+\ez)^{\bz/2 }v_{x_i}-c_i][(|Dv|^2+\ez)^{\bz/2 }v_{x_j}]_{x_i}\psi_{x_j}\,dx.\nonumber
\end{align}
\end{lem}

We also need the following Gehring's lemma; see \cite{MPS2000}.
\begin{lem}\label{lem4}

Let $F$ and $G$ be nonnegative functions in $\Omega$ with $ F \in L^r(\Omega), G \in L^s(\Omega)$ for some $1<r<s$. 
Suppose that  
$$
\dashint_{B}(F(x))^r d x \leq C\left(\dashint_{2B} F(x) d x\right)^r+C \dashint_{2B}(G(x))^r d x,  
$$
 for any ball $B $ with $2B\Subset \Omega$.
Then there exists $\ez \in(0, s-r)$ such that
$$
\left(\dashint_{B}(F(x))^p d x\right)^{1 / p} \leq K\left(\dashint_{2B}(F(x))^r d x\right)^{1 / r}+K\left(\dashint_{2B}(G(x))^p d x\right)^{1 / p}
$$
for every $p \in[r, r+\ez)$, where the constants $K$ and $\ez$ depend only on $C, r, s$ and $n$.

\end{lem}

\section{The estimate of normalized parabolic $p(x)$-Laplacian}
To prove Theorem 1.1, given any fixed smooth domain $U\subset\subset\Omega$, and for $\epsilon\in(0,1]$, let $u^\epsilon$ be a viscosity solution to the regularized equation (1.4), we know that $u^\epsilon$ has the $C^{1,\alpha}_{loc}$-regularity in the
spatial variable and the $C^{1,\frac{1+\alpha}{2}}$-regularity in the time variable, especially $Du^\epsilon\in L^\infty(U_T)$ uniformly in $\epsilon>0$; see\cite{2020Gradient}.

Applying Lemma 2.2, we prove the following.
\begin{lem}\label{lem3-1}
If $n\ge2$, $p(x)\in C^1(\Omega)$ and $p(x)\in(1,3+\frac{2}{n-2})$, then we have
\begin{align*}
&\int_{Q_{2r}}|D^2u^\epsilon|^2\phi^2dxdt+\int_{Q_{2r}}(u_t^\epsilon)^2\phi^2dxdt\\
&\quad \le C(n,p_{+}){\rm inf}_{c\in \mathbb{R}^n}
\int_{Q_{2r}}|Du^\epsilon-c|^2[|D\phi|^2+|\phi||D^2\phi|+|\phi| |\phi_t|]dxdt\\
&\qquad +C(n,p_{+})\epsilon\int_{Q_{2r}}[1+|ln[|Du^\epsilon|^2+\epsilon]|][|D\phi|^2+|Dp|^2+|\phi_t||\phi|]dxdt\\
&\qquad +C(n,p_{+})\epsilon\int_{B_{2r}}|ln[|Du^\epsilon(x,0)|^2+\epsilon]|\phi^2(x,0)dx
\tag{3.1}
\end{align*}
\end{lem}

\begin{lem}\label{lem3-2}

\begin{align}\label{neq4.3}
\dashint_{\frac12Q}|D[(|Dv|^2+\ez)^{\frac{p(x)-2}4}Dv]|^2dx\,dt\le &C(n,p^B_\pm)\frac1{R^2}\dashint_{Q}|(|Dv|^2+\ez)^{\frac{p(x)-2}4}Dv-\vec{c}|^2\,dx\,dt\\
&+C(n,p^B_\pm)\dashint_{Q}(|Dv|^2+\ez)^{\frac{p(x)-1}2}|Df^\ez|\,dx\,dt\nonumber\\
&+C(n,p^B_\pm)(\frac1R+\frac1{R^2})\dashint_{Q}(|Dv|^2+\ez)^{\frac{p(x)}2}\,dx\,dt.\nonumber
\end{align}

\end{lem}
Given Lemma \ref{lem3-1} and \ref{lem3-2}, we prove Theorem 1.1 as below.
\begin{proof}[Proof of Theorem 1.1.]
Lemma 3.1 together with $Du^\epsilon\in L^\infty(U_T)$ uniformly in $\epsilon>0$, implies that $D^2 u^\epsilon$,
$u^\epsilon_t\in L^2_{loc}(U_T)$ uniformly in $\epsilon>0$. By the compact parabolic embedding theorem,
$Du^\epsilon\rightarrow Du$ in $L^2_{loc}(U_T)$, and $D^2 u^\epsilon\rightarrow D^u$ and
$u^\epsilon_t\rightarrow u_t$ weakly in $L^2_{loc}(U_T)$ as $\epsilon\rightarrow 0$. Letting
$\epsilon\rightarrow 0$ and arbitrariness of $U_T$ we finish the proof of Theorem 1.1
\end{proof}
Finally, we prove Lemma 3.1. First, applying Lemma 2.1 to $u^{\epsilon}$, we have
\begin{align*}
\frac{n}{2}|D^2u^{\epsilon}Du^{\epsilon}|^2-\Delta u^{\epsilon}\Delta_{\infty}u^{\epsilon}
-\frac{n-2}{2}(\Delta u^{\epsilon})^2|Du^{\epsilon}|^2
\le\frac{n-1}{2}[|D^2u^{\epsilon}|^2-(\Delta u^{\epsilon})^2]|Du^{\epsilon}|^2.
\end{align*}
Dividing both sides by $|Du^{\epsilon}|^2+\epsilon$, using\eqref{eq1.5} and multiplying both sides by $(p(x,t)-2)^2$
we obtain
\begin{align*}
&\frac{n}{2}(p(x,t)-2)^2\frac{|D^2u^{\epsilon}Du^{\epsilon}|^2}
{|Du^{\epsilon}|^2+\epsilon}+[(p(x,t)-2)-\frac{n-2}{2}(p(x,t)-2)^2](\Delta u^{\epsilon})^2\\
&\quad\le\frac{n-1}{2}(p(x,t)-2)^2[|D^2u^{\epsilon}|^2-(\Delta u^{\epsilon})^2]\\
&\qquad+\frac{\epsilon}{2}(p(x,t)-2)^2\frac{(\Delta u^{\epsilon})^2-|D^2u^{\epsilon}|^2}{|Du^{\epsilon}|^2+\epsilon}
+(p(x,t)-2)\Delta u^{\epsilon}u_t^\epsilon\tag{3.2}
\end{align*}
Moreover, since
\begin{align*}
(p(x,t)-2)^2\frac{|D^2u^{\epsilon}Du^{\epsilon}|^2}{|Du^{\epsilon}|^2+\epsilon}
&\ge(p(x,t)-2)^2[\frac{\Delta_{\infty}u^\epsilon}{|Du^{\epsilon}|^2+\epsilon}]^2\\
&=[u_t^\epsilon-\Delta u^\epsilon]^2\\
&=(u_t^\epsilon)^2+(\Delta u^\epsilon)^2-2u_t^\epsilon \Delta u^\epsilon
\end{align*}
(3.2) leads to
\begin{align*}
[\frac{n}{2}+(p(x,t)-2)-&\frac{n-2}{2}(p(x,t)-2)^2](\Delta u^\epsilon)^2
+\frac{n}{2}(u_t^\epsilon)^2\\
&\le(p(x,t)-2)^2\frac{n-1}{2}[|D^2u^\epsilon|^2-(\Delta u^\epsilon)^2]\\
&\quad+\frac{\epsilon}{2}(p(x,t)-2)^2\frac{(\Delta u^{\epsilon})^2-|D^2u^{\epsilon}|^2}{|Du^{\epsilon}|^2+\epsilon}
+[(p(x,t)-2)+n]\Delta u^\epsilon u_t^\epsilon
\end{align*}
Adding both sides by
\begin{align*}
[\frac{n}{2}+(p(x,t)-2)-\frac{n-2}{2}(p(x,t)-2)^2][|D^2u^\epsilon|^2-(\Delta u^\epsilon)^2]
\end{align*}
we conclude that
\begin{align*}
[\frac{n}{2}+(p(x,t)-2)-&\frac{n-2}{2}(p(x,t)-2)^2]|D^2u^\epsilon|^2+\frac{n}{2}(u_t^\epsilon)^2\\
&\le[\frac{n}{2}+(p(x,t)-2)+\frac{1}{2}(p(x,t)-2)^2][|D^2u^\epsilon|^2-(\Delta u^\epsilon)^2]\\
&\quad+\frac{\epsilon}{2}(p(x,t)-2)^2\frac{(\Delta u^{\epsilon})^2-|D^2u^{\epsilon}|^2}{|Du^{\epsilon}|^2+\epsilon}\\
&\quad+[(p(x,t)-2)+n]\Delta u^\epsilon u_t^\epsilon\tag{3.3}
\end{align*}

Below we consider 2 cases respectively:

$\bullet$ Case $n\ge 3$ and $1<p(x,t)<3+\frac{2}{n-2}$. In this case we use (3.3) to prove (3.1).

$\bullet$ Case $n=2$ and $1<p(x)<\infty$ we use (3.2) to prove (3.1).\\
\large\textbf {Case} \large{$1<p(x,t)<3+\frac{2}{n-2}$:}\\
Via a direct calculation we have the following
\begin{lem}
Let $1<p_{-}\le p_{+}<\infty$ and $p(x,t)\in C^{1}(\overline{\Omega})$. For any $c\in \mathbb{R}^n$, we have
\begin{align*}
\int_{Q_{2r}}\Delta u^\epsilon u_t^\epsilon \phi^2dxdt
&\le \eta \int_{Q_{2r}}(p(x,t)+n)^2|D^2u^\epsilon|^2\phi^2dxdt\\
&\quad+\frac{C(n,p_{+})}{\eta}\int_{Q_{2r}}|Du^\epsilon-c|^2[|D\phi|^2+|\phi\phi_t|]dxdt\tag{3.4}
\end{align*}
\end{lem}
\begin{proof}
By integration by parts we have
\begin{align*}
\int_{Q_{2r}}\Delta u^\epsilon u_t^\epsilon \phi^2dxdt
&=\int_{Q_{2r}}(u_{x_i}^\epsilon-c_i)_{x_i}u_t^\epsilon \phi^2dxdt\\
&=-\int_{Q_{2r}}(u_{x_i}^\epsilon-c_i)u_{x_{i}t}^\epsilon \phi^2dxdt
-2\int_{Q_{2r}}(u_{x_i}^\epsilon-c_i)u_t^\epsilon \phi\phi_{x_i}dxdt.
\end{align*}
Further integration by parts gives
\begin{align*}
-\int_{Q_{2r}}(u_{x_i}^\epsilon-c_i)u_{x_{i}t}^\epsilon \phi^2dxdt
=-\frac{1}{2}\int_{Q_{2r}}(|Du^\epsilon-c|^2)_t\phi^2dxdt\le\int_{Q_{2r}}|Du^\epsilon-c|^2 |\phi| |\phi_t|dxdt
\end{align*}
and
\begin{align*}
-2\int_{Q_{2r}}(u_{x_i}^\epsilon-c_i)u_t^\epsilon \phi\phi_{x_i}dxdt
\le&2\int_{Q_{2r}}|Du^\epsilon-c||u_t^\epsilon||D\phi||\phi|dxdt\\
&\frac{1}{\eta}\int_{Q_{2r}}|Du^\epsilon-c|^2|D\phi|^2dxdt
+\eta \int_{Q_{2r}}|u_t^\epsilon|^2\phi^2dxdt
\end{align*}
By the equation (1.5), we have
\begin{align*}
|u_t^\epsilon| &\le |\Delta u^\epsilon|+(p(x,t)-2)|D^2u^\epsilon|\le (p(x)+n)|D^2u^\epsilon|\tag{3.5}
\end{align*}
Combining all the estimations, we have (3.4) as desired.
\end{proof}
Using this and the divergence structure of $[|D^2u^\epsilon|^2-(\Delta u^\epsilon)^2]$, we further have the following.
\begin{lem}
Let $1<p_{-}\le p_{+}<\infty$ and $p(x,t)\in C^{1}(\overline{\Omega})$. For any $\eta \in (0,1)$, we have
\begin{align*}
&\frac{\epsilon}{2}\int_{Q_{2r}} (p(x,t)-2)^2
\frac{(\Delta u^\epsilon)^2-|D^2u^{\epsilon}|^2}{|Du^{\epsilon}|^2+\epsilon}\phi^2dxdt\\
&\quad\le\int_{Q_{2r}}\frac{1}{4}\frac{(p(x,t)-2)^2}{(p(x,t)-1)}(u_t^\epsilon)^2\phi^2dxdt\\
&\qquad+\eta\int_{Q_{2r}}|D^2u^\epsilon|^2\phi^2dxdt
+\epsilon \frac{C(p_{+},n)}{\eta} \int_{Q_{2r}}|D\phi|^2dxdt\tag{3.6}
\end{align*}
\end{lem}
\begin{proof}
By integration by parts, we obtain
\begin{align*}
\frac{\epsilon}{2}&\int_{Q_{2r}} (p(x,t)-2)^2
\frac{(\Delta u^\epsilon)^2-|D^2u^{\epsilon}|^2}{|Du^{\epsilon}|^2+\epsilon}\phi^2dxdt\\
&=-\frac{\epsilon}{2}\int_{Q_{2r}}[\Delta u^\epsilon u^\epsilon_{x_i}-u^\epsilon_{x_ix_j}u^\epsilon_{x_j}]
\left(\frac{(p(x,t)-2)^2\phi^2} {|Du^\epsilon|^2+\epsilon} \right)_{x_i}dxdt\\
&=\epsilon \int_{Q_{2r}}(p(x,t)-2)^2 \left( \Delta u^\epsilon\frac{\Delta_\infty u^\epsilon}{[|Du^\epsilon|^2+\epsilon]^2}
-\frac{|D^2u^\epsilon Du^\epsilon|^2}{[|Du^\epsilon|^2+\epsilon]^2}\right)\phi^2dxdt\\
&\quad-\epsilon \int_{Q_{2r}}[\Delta u^\epsilon u^\epsilon_{x_i}-u^\epsilon_{x_ix_j}u^\epsilon_{x_j}]\phi_{x_i}\phi
\frac{(p(x,t)-2)^2}{|Du^\epsilon|^2+\epsilon}dxdt\\
&\quad-\epsilon \int_{Q_{2r}}[\Delta u^\epsilon u^\epsilon_{x_i}-u^\epsilon_{x_ix_j}u^\epsilon_{x_j}]\phi^2
\frac{(p(x,t)-2)p_{x_i}}{|Du^\epsilon|^2+\epsilon}dxdt
\end{align*}
By Young's inequality we obtain
\begin{align*}
\epsilon &\int_{Q_{2r}}[\Delta u^\epsilon u^\epsilon_{x_i}-u^\epsilon_{x_ix_j}u^\epsilon_{x_j}]\phi_{x_i}\phi
\frac{1}{|Du^\epsilon|^2+\epsilon}dxdt\\
&\le C \epsilon^{\frac{1}{2}}\int_{Q_{2r}} |D^2u^\epsilon| |\phi D\phi|dxdt\\
&\le \eta \int_{Q_{2r}}|D^2u^\epsilon|^2\phi^2dxdt+\epsilon C(\eta)\int_{Q_{2r}} |D\phi|^2dxdt
\end{align*}
Similarly,
\begin{align*}
\epsilon &\int_{Q_{2r}}[\Delta u^\epsilon u^\epsilon_{x_i}-u^\epsilon_{x_ix_j}u^\epsilon_{x_j}]\phi^2
\frac{(p(x,t)-2)p_{x_i}}{|Du^\epsilon|^2+\epsilon}dxdt\\
&\le \eta \int_{Q_{2r}}|D^2u^\epsilon|^2\phi^2dxdt+\epsilon C(\eta)\int_{Q_{2r}} |Dp|^2\phi^2dxdt
\end{align*}
By H$\ddot{\rm o}$lder's inequality, (1.5), and Young's inequality one has
\begin{align*}
\epsilon&(p(x,t)-2)^2 \left( \Delta u^\epsilon\frac{\Delta_\infty u^\epsilon}{[|Du^\epsilon|^2+\epsilon]^2}
-\frac{|D^2u^\epsilon Du^\epsilon|^2}{[|Du^\epsilon|^2+\epsilon]^2}\right)\\
&\le\epsilon(p(x,t)-2)^2 \left( \Delta u^\epsilon\frac{\Delta_\infty u^\epsilon}{[|Du^\epsilon|^2+\epsilon]^2}
-\frac{(\Delta_\infty u^\epsilon)^2}{[|Du^\epsilon|^2+\epsilon]^3}\right)\\
&=\frac{\epsilon}{|Du^\epsilon|^2+\epsilon}
\left[(p(x,t)-2)\Delta u^\epsilon(u_t^\epsilon-\Delta u^\epsilon)
-(u_t^\epsilon-\Delta u^\epsilon)^2 \right]\\
&\le\frac{(p(x,t)-2)^2}{4(p(x,t)-1)}(u_t^\epsilon)^2
\end{align*}
Combining all estimates together, we get (3.6).
\end{proof}
Then we give the proof of Lemma 3.1 in the case of $n\ge 3$ and $1<p(x,t)<3+\frac{2}{n-2}$
\begin{proof}
By Lemma(2.2), for any $c\in \mathbb{R}^n$ one gets
\begin{align*}
|\int_{Q_{2r}}[|D^2u^\epsilon|^2-(\Delta u^\epsilon)^2]\phi^2dxdt|
\le C(n)\int_{Q_{2r}}|Du^\epsilon-c|[|D\phi|^2+|D^2\phi||D\phi|]dxdt\tag{3.7}
\end{align*}
Multiplying both sides of (3.3) by $\phi^2$ and integrating, by (3.7), Lemma 3.2 and Lemma 3.3, for any $c\in \mathbb{R}^n$ and
$\eta \in (0,1)$ we obtain
\begin{align*}
\int_{Q_{2r}}&\left[\frac{n}{2}+(p(x,t)-2)-\frac{n-2}{2}(p(x,t)-2)^2-\eta\right]|D^2u^\epsilon|^2\phi^2dxdt\\
&\quad+\int_{Q_{2r}}\left[\frac{n}{2}-\frac{1}{4}\frac{(p(x,t)-2)^2}{(p(x,t)-1)}\right]\int_{Q_{2r}}(u_t^\epsilon)^2\phi^2dxdt\\
&\le C(n,p_{+},\eta)\int_{Q_{2r}}|Du^\epsilon-c|^2[|D\phi|^2+|\phi\phi_t|]dxdt
+\epsilon C(\eta)\int_{Q_{2r}}(|D\phi|^2+|Dp|^2)dxdt
\end{align*}
Note that $p(x,t)\in (1,3+\frac{2}{n-2})$ implies that
\begin{align*}
\frac{n}{2}+(p(x,t)-2)-\frac{n-2}{2}(p(x,t)-2)^2>0
\end{align*}
Moreover, when $p(x,t)\in (2+n-\sqrt{n^2+2n},n+2+\sqrt{n^2+2n})$, we have
\begin{align*}
\frac{n}{2}+(p(x,t)-2)-\frac{n-2}{2}(p(x,t)-2)^2>0
\end{align*}
Taking $\eta>0$ sufficiently small, and noting $n+2+\sqrt{n^2+2n}\ge 6$, one has (3.1) under the condition that
$p(x,t)\in(n+2-\sqrt{n^2+2n},3+\frac{2}{n-2})$. Finally, it remains to notice by adding dummy variable $u$
also satisfies the equation in $\mathbb{R}^m$ for any $m\ge n$. Therefore, (3.1) holds for $p(x,t)$ in
$$ \bigcup_{m\ge n}(n+2-\sqrt{n^2+2n},3+\frac{2}{n-2})=(1,3+\frac{2}{n-2})$$
The lemma is proved in this case.
\end{proof}
\large\textbf{Case} $n=2$ \large\textbf{and}\large{$1<p(x,t)<\infty$}\\
Instead of Lemma 3.2, we have the following.
\begin{lem}
Let $n=2$ and $1<p(x,t)<\infty$. For any $\eta\in(0,1)$ we have
\begin{align*}
\frac{\epsilon}{2}&\int_{Q_{2r}}(p(x,t)-2)^2
\frac{(\Delta u^\epsilon)^2-|D^2u^{\epsilon}|^2}{|Du^{\epsilon}|^2+\epsilon}\phi^2dxdt\\
&\le\epsilon\int_{Q_{2r}}(p(x,t)-2)\frac{\Delta u^\epsilon u_t^\epsilon}{|Du^\epsilon|^2+\epsilon}\phi^2dxdt\\
&\quad+\eta\int_{Q_{2r}}|D^2 u^\epsilon|^2\phi^2dxdt
+\epsilon\frac{C}{\eta}\int_{Q_{2r}}(|D\phi|^2+|Dp|^2)dxdt
\end{align*}
\end{lem}
\begin{proof}
The proof follows that of Lemma 3.3 once we observe that
\begin{align*}
\epsilon \int_{Q_{2r}}&(p(x,t)-2)^2 \Delta u^\epsilon \frac{\Delta_\infty u^\epsilon}{[|Du^\epsilon|^2+\epsilon]^2}\phi^2dxdt\\
&\le\epsilon  \int_{Q_{2r}}(p(x,t)-2)\frac{u_t^\epsilon \Delta u^\epsilon}{|Du^\epsilon|^2+\epsilon}\phi^2dxdt\\
&\quad +\epsilon\int_{Q_{2r}}|p(x,t)-2|\frac{(\Delta u^\epsilon)^2}{|Du^\epsilon|^2+\epsilon}\phi^2dxdt
\end{align*}
\end{proof}

\begin{lem}
Let $n=2$ and $1<p(x,t)<\infty$. For any $\eta>0$, we have
\begin{align*}
2\epsilon &\int_{Q_{2r}}(p(x,t)-2)\frac{u_t^\epsilon \Delta u^\epsilon}{|Du^\epsilon|^2+\epsilon}\phi^2dxdt\\
&\le \int_{Q_{2r}} (u_t^\epsilon)^2\phi^2dxdt
+\epsilon\int_{Q_{2r}}(p(x,t)-2)^2\frac{|D^2u^\epsilon Du^\epsilon|^2}{[|Du^\epsilon|^2+\epsilon]^2}\phi^2dxdt
+\eta\int_{Q_{2r}}|D^2u^\epsilon|^2\phi^2dxdt\\
&\quad+\epsilon C\int_{Q_{2r}}|ln(|Du^\epsilon|^2+\epsilon)||\phi \phi_t|dxdt
+\epsilon \frac{C}{\eta}\int_{Q_{2r}}(|D\phi|^2+|Dp|^2)\phi^2dxdt\\
&\quad+\epsilon\int_{B_{2r}}|ln(|Du^\epsilon(x,0)|^2+\epsilon)| |\phi^2(x,0)|dx\tag{3.8}
\end{align*}
\end{lem}
\begin{proof}
By integration by parts we have
\begin{align*}
2\epsilon &\int_{Q_{2r}}(p(x,t)-2)\frac{u_t^\epsilon \Delta u^\epsilon}{|Du^\epsilon|^2+\epsilon}\phi^2dxdt\\
&=-2\epsilon\int_{Q_{2r}} u^\epsilon_{x_i}
\left( \frac{(p(x,t)-2)u_t^\epsilon}{|Du^\epsilon|^2+\epsilon}\phi^2 \right)_{x_i}dxdt\\
&=-2\epsilon\int_{Q_{2r}}(p(x,t)-2)\frac{u_{x_it}^\epsilon u_{x_i}^\epsilon}{|Du^\epsilon|^2+\epsilon}\phi^2dxdt
-4\epsilon\int_{Q_{2r}}(p(x,t)-2)\frac{u_{x_i}^\epsilon u_t^\epsilon \phi_{x_i}}{|Du^\epsilon|^2+\epsilon}\phi dxdt\\
&\quad\ +4\epsilon\int_{Q_{2r}}(p(x,t)-2)\frac{\Delta u_\infty^\epsilon u_t^\epsilon}{[|Du^\epsilon|^2+\epsilon]^2}\phi^2dxdt
-2\epsilon\int_{Q_{2r}}\frac{p_i u_i^\epsilon u_t^\epsilon}{|Du^\epsilon|^2+\epsilon}\phi^2dxdt\\
&=I_1+I_2+I_3+I_4
\end{align*}
For $I_1$, from
\begin{align*}
\frac{u_{x_it}^\epsilon u_{x_i}^\epsilon}{|Du^\epsilon|^2+\epsilon}
=\frac{1}{2}\frac{(|Du^\epsilon|^2)_t}{|Du^\epsilon|^2+\epsilon}
=\frac{1}{2}[ln(|Du^\epsilon|^2+\epsilon)]_t\tag{3.9}
\end{align*}
and integration by parts it follows that
\begin{align*}
-2\epsilon &\int_{Q_{2r}}(p(x,t)-2)\frac{u_{x_it}^\epsilon u_{x_i}^\epsilon}{|Du^\epsilon|^2+\epsilon}\phi^2dxdt\\
&=-2\epsilon \int_{Q_{2r}}(p(x,t)-2)[ln(|Du^\epsilon|^2+\epsilon)]_t\phi^2dxdt\\
&=2\epsilon\int_{Q_{2r}}(p(x,t)-2)ln(|Du^\epsilon|^2+\epsilon)\phi \phi_tdxdt\\
&\quad+\epsilon\int_{Q_{2r}}p_t(x,t)ln(|Du^\epsilon|^2+\epsilon)\phi^2dxdt\\
&\quad-\epsilon\int_{B_{2r}}(p(x,t)-2)|ln(|Du^\epsilon(x,0)|^2+\epsilon)| |\phi^2(x,0)|dx
\end{align*}
For $I_2$, by Young's inequality and (3.5), one has
\begin{align*}
|4\epsilon\int_{Q_{2r}}&(p(x,t)-2)\frac{u_{x_i}^\epsilon u_t^\epsilon \phi_{x_i}}{|Du^\epsilon|^2+\epsilon}\phi dxdt|\\
\le&\eta\int_{Q_{2r}}|D^2u^\epsilon|^2\phi^2dxdt
+\epsilon \frac{C}{\eta} \int_{Q_{2r}} |D\phi|^2dxdt\tag{3.10}
\end{align*}
For $I_3$, by Young's inequality and noting
\begin{align*}
4\epsilon|Du^\epsilon|^2=(2\epsilon^{\frac{1}{2}}|Du^\epsilon|)^2\le[|Du^\epsilon|^2+\epsilon]^2
\end{align*}
we have
\begin{align*}
4\epsilon&\int_{Q_{2r}}(p(x,t)-2)\frac{\Delta u_\infty^\epsilon u_t^\epsilon}{[|Du^\epsilon|^2+\epsilon]^2}\phi^2dxdt\\
&\le 4\epsilon^2 \int_{Q_{2r}} (p(x,t)-2)^2\frac{|D^2u^\epsilon Du^\epsilon|^2|Du^\epsilon|^2}{[|Du^\epsilon|^2+\epsilon]^4}\phi^2dxdt
+\int_{Q_{2r}}(u_t^\epsilon)^2\phi^2dxdt\\
&\le \epsilon\int_{Q_{2r}}(p(x,t)-2)^2 \frac{|D^2u^\epsilon Du^\epsilon|^2}{[|Du^\epsilon|^2+\epsilon]^2}\phi^2dxdt
+\int_{Q_{2r}} (u_t^\epsilon)^2\phi^2dxdt\tag{3.11}
\end{align*}
For $I_4$, using (1.5) and Young's inequality and noting
$$|\frac{\Delta_\infty u^\epsilon}{|Du^\epsilon|^2+\epsilon}|\le|D^2 u^\epsilon|$$
we have
\begin{align*}
-2\epsilon&\int_{Q_{2r}}\frac{p_i u_i^\epsilon u_t^\epsilon}{|Du^\epsilon|^2+\epsilon}\phi^2dxdt\\
&=-2\epsilon\int_{Q_{2r}}\frac{p_i u^\epsilon_{x_i}}{|Du^\epsilon|^2+\epsilon}
(\Delta u^\epsilon+(p(x,t)-2)\frac{\Delta_\infty u^\epsilon}{|Du^\epsilon|^2+\epsilon})\phi^2dxdt\\
&\le\eta \int_{Q_{2r}}|D^2 u^\epsilon|^2\phi^2dxdt+\epsilon\frac{C}{\eta}\int_{Q_{2r}}|Dp|^2\phi^2dxdt\tag{3.12}
\end{align*}
Combining these estimates, we get (3.8)
\end{proof}

\begin{lem}
Let $n=2$ and $1<p(x,t)<\infty$. For any $\eta>0$, we have
\begin{align*}
\int_{Q_{2r}}&(p(x,t)-2)\Delta u^\epsilon u_t^\epsilon\frac{|Du^\epsilon|^2}{|Du^\epsilon|^2+\epsilon}\phi^2dxdt\\
&\le-\int_{Q_{2r}}(u_t^\epsilon)\phi^2dxdt
-\epsilon\int_{Q_{2r}}(p(x,t)-2)^2\frac{|D^2u\epsilon Du^\epsilon|^2}{[|Du^\epsilon|^2+\epsilon]^2}\phi^2dxdt\\
&\quad+\int_{Q_{2r}}(p(x,t)-2)^2\frac{|D^2u\epsilon Du^\epsilon|^2}{|Du^\epsilon|^2+\epsilon}\phi^2dxdt\\
&\quad+\eta\int_{Q_{2r}}|D^2u^\epsilon|^2\phi^2dxdt
+C\int_{Q_{2r}}|Du^\epsilon-c|^2[|D\phi|^2+|\phi| |D^2\phi|+|\phi| |\phi_t|]dxdt\\
&\quad+\frac{C}{\eta}\int_{Q_{2r}}(|D\phi|^2+|Dp|^2)dxdt
+\frac{\epsilon}{4}\int_{Q_{2r}}|ln(|Du^\epsilon|^2+\epsilon)| |\phi| |\phi_t|dxdt\\
&\quad+\frac{\epsilon}{8}\int_{B_{2r}}ln(|Du^\epsilon(x,0)|^2+\epsilon)\phi^2(x,0)dx\tag{3.13}
\end{align*}
\end{lem}
\begin{proof}
By integration by parts, one gets
\begin{align*}
\int_{Q_{2r}}&(p(x,t)-2) u_t^\epsilon \Delta u^\epsilon \phi^2 \frac{|Du^\epsilon|^2}{|Du^\epsilon|^2+\epsilon}dxdt\\
&=-\int_{Q_{2r}} u^\epsilon_{x_i}
\left(\frac{(p(x,t)-2)u_t^\epsilon \phi^2 |Du^\epsilon|^2}{|Du^\epsilon|^2+\epsilon}\right)_{x_i}dxdt\\
&=-2\int_{Q_{2r}}\frac{(p(x,t)-2)\Delta_\infty u^\epsilon u^\epsilon_t \phi^2}{|Du^\epsilon|^2+\epsilon}dxdt
-\int_{Q_{2r}}\frac{(p(x,t)-2)u_{x_i}^\epsilon u_{tx{_i}}^\epsilon \phi^2 |Du^\epsilon|^2}{|Du^\epsilon|^2+\epsilon}dxdt\\
&\quad-2\int_{Q_{2r}}\frac{(p(x,t)-2)u_{x_i}^\epsilon u_{t}^\epsilon \phi \phi_{x_i}|Du^\epsilon|^2}{|Du^\epsilon|^2+\epsilon}dxdt
+2\int_{Q_{2r}}\frac{(p(x,t)-2)\Delta_{\infty}u^\epsilon u_t^\epsilon \phi^2 |Du^\epsilon|^2}{[|Du^\epsilon|^2+\epsilon]^2}dxdt\\
&\quad-\int_{Q_{2r}}\frac{u^\epsilon_{x_i}p_i u_t^\epsilon \phi^2 |Du^\epsilon|^2}{[|Du^\epsilon|^2+\epsilon]^2}dxdt\\
&=I_1+I_2+I_3+I_4+I_5
\end{align*}
For $I_1$, from Lemma 3.2 it follows that
\begin{align*}
-2\int_{Q_{2r}}\frac{(p(x,t)-2)\Delta_{\infty} u^\epsilon u^\epsilon_t \phi^2}{|Du^\epsilon|^2+\epsilon}dxdt
&=-2\int_{Q_{2r}}(u_t^\epsilon)^2\phi^2dxdt
+2\int_{Q_{2r}}\Delta u^\epsilon u_t^\epsilon \phi^2dxdt\\
&\le -2\int_{Q_{2r}}(u_t^\epsilon)^2\phi^2dxdt+\eta\int_{Q_{2r}}|D^2u^\epsilon|^2\phi^2dxdt\\
&\quad+\frac{C}{\eta}\int_{Q_{2r}}|Du^\epsilon-c|^2[|D\phi|^2+|\phi \phi_t|]dxdt
\end{align*}
For $I_2$, by (3.9)
\begin{align*}
-&\int_{Q_{2r}}\frac{(p(x,t)-2)u_{x_i}^\epsilon u_{tx_{i}}^\epsilon \phi^2 |Du^\epsilon|^2}{|Du^\epsilon|^2+\epsilon}dxdt\\
&=-\int_{Q_{2r}}(p(x,t)-2)u_{x_i}^\epsilon u_{tx{_i}}^\epsilon \phi^2dxdt
+\epsilon\int_{Q_{2r}} (p(x,t)-2)ln(|Du^\epsilon|^2+\epsilon) |\phi \phi_t|dxdt\\
&\quad+\frac{\epsilon}{2}\int_{B_{2r}}(p(x,t)-2) ln[|Du^\epsilon(x,0)|^2+\epsilon]\phi^2(x,0)dx
\end{align*}
For $I_3$, by integration by parts
\begin{align*}
-2&\int_{Q_{2r}}\frac{(p(x,t)-2)u_{x_i}^\epsilon u_{t}^\epsilon \phi \phi_{x_i}|Du^\epsilon|^2}{|Du^\epsilon|^2+\epsilon}dxdt\\
&=-\int_{Q_{2r}}(p(x,t)-2)u_{x_i}^\epsilon u_{t}^\epsilon (\phi^2)_{x_i}dxdt
+2\epsilon\int_{Q_{2r}}(p(x,t)-2)\frac{u_{x_i}^\epsilon u_{t}^\epsilon \phi \phi_{x_i}}{|Du^\epsilon|^2+\epsilon}dxdt\\
&=\int_{Q_{2r}}(p(x,t)-2)u_{x_i}^\epsilon u_{tx_{i}}^\epsilon \phi^2dxdt
+\int_{Q_{2r}}(p(x,t)-2)\Delta u^\epsilon u_{t}^\epsilon \phi^2dxdt\\
&\quad+\int_{Q_{2r}} p_i u_{x_i}u^\epsilon_t \phi^2dxdt
+2\epsilon\int_{Q_{2r}}\frac{(p(x,t)-2)u_{x_i}^\epsilon u_{t}^\epsilon \phi \phi_{x_i}}{|Du^\epsilon|^2+\epsilon}dxdt.
\end{align*}
Applying Lemma 2.2 and (3.10), we get
\begin{align*}
-2&\int_{Q_{2r}}\frac{(p(x,t)-2)u_{x_i}^\epsilon u_{t}^\epsilon \phi \phi_{x_i}|Du^\epsilon|^2}{|Du^\epsilon|^2+\epsilon}dxdt\\
&\le\int_{Q_{2r}}(p(x,t)-2)u_{x_i}^\epsilon u_{tx{_i}}^\epsilon \phi^2dxdt
+C\int_{Q_{2r}}|Du^\epsilon-c|^2[|D\phi|^2+|\phi| |D^2\phi|+|\phi| |\phi_t|]dxdt\\
&\quad+\eta\int_{Q_{2r}}|D^2u^\epsilon|^2\phi^2dxdt+\epsilon\frac{C}{\eta}\int_{Q_{2r}}(|D\phi|^2+|Dp|^2)dxdt
\end{align*}
For $I_4$, by Holder's inequality and Young's inequality we obtain
\begin{align*}
2\int_{Q_{2r}}(p(x,t)&-2)\frac{\Delta_{\infty}u^\epsilon u_t^\epsilon \phi^2 |Du^\epsilon|^2}{[|Du^\epsilon|^2+\epsilon]^2}dxdt\\
&\le2\int_{Q_{2r}}|p(x,t)-2|\frac{|D^2u^\epsilon Du^\epsilon| |Du^\epsilon| |u_t^\epsilon|\phi^2}{|Du^\epsilon|^2+\epsilon}dxdt\\
&\le\int_{Q_{2r}}(u_t^\epsilon)^2\phi^2dxdt
+\int_{Q_{2r}}\frac{(p(x,t)-2)^2|D^2u^\epsilon Du^\epsilon|^2 |Du^\epsilon|^2}{[|Du^\epsilon|^2+\epsilon]^2}\phi^2dxdt\\
&=\int_{Q_{2r}}(u_t^\epsilon)^2\phi^2dxdt
+\int_{Q_{2r}}\frac{(p(x,t)-2)^2|D^2u^\epsilon Du^\epsilon|^2 }{|Du^\epsilon|^2+\epsilon}\phi^2dxdt\\
&\qquad \qquad \qquad\qquad \
-\epsilon\int_{Q_{2r}}\frac{(p(x,t)-2)^2|D^2u^\epsilon Du^\epsilon|^2}{[|Du^\epsilon|^2+\epsilon]^2}\phi^2dxdt
\end{align*}
For $I_5$, similar to (3.12), we get
\begin{align*}
-\epsilon&\int_{Q_{2r}}\frac{p_i u_i^\epsilon u_t^\epsilon|Du^\epsilon|^2}{|Du^\epsilon|^2+\epsilon}\phi^2dxdt\\
&\le\eta \int_{Q_{2r}}|D^2 u^\epsilon|^2\phi^2dxdt+\frac{C}{\eta}\int_{Q_{2r}}|Dp|^2|Du^\epsilon|^2\phi^2dxdt
\end{align*}
Combining above, we conclude (3.13)
\end{proof}

Follow Lemma 3.5 and Lemma 3.6, we get
\begin{lem}
Let $n=2$ and $1< p(x,t)<\infty$. For any $c\in \mathbb{R}^n$ and $\eta\in(0,1)$ we have
\begin{align*}
\int_{Q_{2r}}&(p(x,t)-2)\Delta u^\epsilon u_t^\epsilon \phi^2dxdt\\
&\le-\epsilon\int_{Q_{2r}}(p(x,t)-2)\frac{\Delta u^\epsilon u_t^\epsilon}{|Du^\epsilon|^2+\epsilon}\phi^2dxdt
+\int_{Q_{2r}}(p(x,t)-2)^2\frac{|D^2u^\epsilon Du^\epsilon|^2}{|Du^\epsilon|^2+\epsilon}\phi^2dxdt\\
&\quad+\eta\int_{Q_{2r}}|D^2u^\epsilon|^2\phi^2dxdt
+C\int_{Q_{2r}}|Du^\epsilon-c|^2[|D\phi|^2+|\phi| |D^2\phi|+|\phi| |\phi_t|]dxdt\\
&\quad+\frac{C}{\eta}\int_{Q_{2r}}(|D\phi|^2+|Dp|^2)dxdt
+\frac{\epsilon}{4}\int_{Q_{2r}}|ln(|Du^\epsilon|^2+\epsilon)| |\phi| |\phi_t|dxdt\\
&\quad+\frac{\epsilon}{8}\int_{B_{2r}}ln(|Du^\epsilon(x,0)|^2+\epsilon)\phi^2(x,0)dx
\end{align*}
\end{lem}
Now we prove Lemma 3.1 in case of $n=2$ and $1< p(x,t)<\infty$
\begin{proof}
Multiplying both sides of (3.2) with $n=2$ by $\phi^2$ and integrating in $Q_{2r}$, by (3.7) Lemma 3.4 and Lemma 3.7,
for any $c\in\mathbb{R}^n$ and $\eta\in(0,1)$ we obtain
\begin{align*}
\int_{Q_{2r}}&(p(x,t)-2)(\Delta u^\epsilon)^2\phi^2dxdt
-\eta\int_{Q_{2r}}|D^2u^\epsilon|^2\phi^2dxdt\\
&\le \int_{Q_{2r}}(u_t^\epsilon)^2\phi^2dxdt+C\int_{Q_{2r}}|Du^\epsilon-c|^2[|D\phi|^2+|\phi| |D^2\phi|+|\phi| |\phi_t|]dxdt\\
&\quad+\epsilon\frac{C}{\eta}\int_{Q_{2r}}|D\phi|^2dxdt
+C\epsilon\int_{Q_{2r}}|ln(|Du^\epsilon|^2+\epsilon)| |\phi| |\phi_t|dxdt\\
&\quad+C\epsilon\int_{B_{2r}}ln(|Du^\epsilon(x,0)|^2+\epsilon)\phi^2(x,0)dx
\end{align*}
Choosing $\eta$ be sufficiently small, such that
$$0<\eta<\frac{(p(x,t)-2)}{2}$$
adding both sides by
$$\int_{Q_{2r}}(p(x,t)-2)\left[|D^2u^\epsilon|^2-(\Delta u^\epsilon)^2\right]\phi^2dxdt$$
and applying (3.7) and Lemma 2.3, we get
\begin{align*}
\int_{Q_{2r}}|D^2 u^\epsilon|^2\phi^2dxdt
\le &C\int_{Q_{2r}}|Du^\epsilon-c|^2[|D\phi|^2+|\phi| |D^2\phi|+|\phi| |\phi_t|]dxdt\\
+&\frac{C}{\eta}\int_{Q_{2r}}(|D\phi|^2+|Dp|^2)dxdt
+C\epsilon\int_{Q_{2r}}|ln(|Du^\epsilon|^2+\epsilon)| |\phi| |\phi_t|dxdt\\
+&C\epsilon\int_{B_{2r}}ln(|Du^\epsilon(x,0)|^2+\epsilon)\phi^2(x,0)dx
\end{align*}
as desired
\end{proof}

In this section, we assume that  $p(x)$ and $f$ are smooth in $\Omega$. Given any $B\Subset\Omega$ and  $\ez\in(0,1)$, let $v$ be a smooth solution to
\begin{align} \label{eq4.1}
v_t-(\Delta v+(p(x)-2) \frac{\Delta_\fz v}{\left|D v\right|^{2}+\ez})=f^\ez \quad {\rm in}\ Q.
\end{align}

We establish the folowing  pointwise $\sigma_2$-estimates. 
\begin{lem}\label{lem4-1}

\begin{align}\label{neq4.2}
(|Dv|^2+\ez)^{\frac{p(x)-2}2}|D^2v|^2\le C(p^B_\pm,n)\sigma_2(D[(|Dv|^2+\ez)^{\frac{p(x)-2}4}Dv])+C(p^B_\pm,n)(|Dv|^2+\ez)^{\frac{p(x)-2}2}(v_t-f^\ez)^2\quad \mbox{in $Q$}. 
\end{align}
\begin{proof}
Since $p(x)>1>2-n$, we have $\frac{p(x)-2}2>-1+\frac{(n-2)(p(x)-1)}{2(n-1)}$.
Through a calculation process similar to \cite[Lemma 3.1]{}, we can get \eqref{neq4.2}.
\end{proof}
\end{lem}

This, together with Lemma 2.3, allows us to get the following.

\begin{proof}[Proof of Lemma \ref{lem3-2}]
Choosing $ \phi\in C^\fz(Q)$ so that
$$\mbox{  $0\le \phi\le 1$ in $Q$, $\phi=1$ in $\frac Q2$, $|D\phi|\le 8/R$ in $Q$ and $|\phi_t|\le 1/R$ in $Q$.}$$ 
Integrating the function over $Q$, we have
\begin{align*}
\int_Q&|D[(|Dv|^2+\ez)^{\frac{p(x)-2}4}Dv]|^2\phi^2\,dx\,dt\\
\le&\int_QC(p^B_\pm,n)\sigma_2(D[(|Dv|^2+\ez)^{\frac{p(x)-2}4}Dv])\phi^2\,dx\,dt+C(p^B_\pm,n)\int_Q(|Dv|^2+\ez)^{\frac{p(x)-2}2}(v_t-f^\ez)^2\phi^2\,dx\,dt.
\end{align*}

Applying Lemma 2.3, one has
\begin{align*}
\int_Q&|D[(|Dv|^2+\ez)^{\frac{p(x)-2}4}Dv]|^2\phi^2\,dx\\
&\le C(n,p^B_\pm)\int_Q|D[(|Dv|^2+\ez)^{{\frac{p(x)-2}4}}Dv]||(|Dv|^2+\ez)^{\frac{p(x)-2}4}Dv-\vec{c}||\phi D\phi|\,dx\,dt\\
&\ \ \ \ +C(n,p^B_\pm)\int_Q(|Dv|^2+\ez)^{\frac{p(x)-2}2}(v_t-f^\ez)^2\phi^2\,dx\,dt.
\end{align*}

Using \eqref{eq4.1}, by integration by parts, we obtain
\begin{align*}
\int_Q&(|Dv|^2+\ez)^{\frac{p(x)-2}2}(v_t-f^\ez)^2\phi^2\,dx\,dt\\
&= \int_Q(|Dv|^2+\ez)^{\frac{p(x)-2}2}(\Delta v+(p(x)-2) \frac{\Delta_\fz v}{\left|D v\right|^{2}+\ez})(v_t-f^\ez)\phi^2\,dx\,dt\\
&= \int_Q {\rm div}((|Dv|^2+\ez)^{\frac{p(x)-2}2}Dv)(v_t-f^\ez)\phi^2\,dx\,dt\\
&=-\int_Q(|Dv|^2+\ez)^{\frac{p(x)-2}2}Dv(v_{tx_i}-f_{x_i})\phi^2\,dx\,dt-2\int_Q(|Dv|^2+\ez)^{\frac{p(x)-2}2}Dv(v_t-f)\phi\phi_{x_i}\,dx\,dt.
\end{align*}

By integration by parts,
\begin{align*}
-\int_Q&(|Dv|^2+\ez)^{\frac{p(x)-2}2}Dv(v_{tx_i}-f_{x_i})\phi^2\,dx\,dt\\
&=-\int_Q\frac1{p(x)}[(|Dv|^2+\ez)^{\frac{p(x)}2}]_t\phi^2\,dx\,dt+\int_Q(|Dv|^2+\ez)^{\frac{p(x)-2}2}Dvf^\ez_{x_i}\phi^2\,dx\,dt\\
&\le\frac2{p_-}\int_Q(|Dv|^2+\ez)^{\frac{p(x)}2}|\phi||\phi_t|\,dx\,dt+\int_Q(|Dv|^2+\ez)^{\frac{p(x)-1}2}|Df^\ez||\phi|^2\,dx\,dt.
\end{align*}

Using Young's inequality,
\begin{align*}
\int_Q&(|Dv|^2+\ez)^{\frac{p(x)-2}2}Dv(v_t-f)\phi\phi_{x_i}\,dx\,dt\\
&\le \frac12\int_Q(|Dv|^2+\ez)^{\frac{p(x)-2}2}(v_t-f)^2\phi^2\,dx\,dt+\frac12\int_Q(|Dv|^2+\ez)^{\frac{p(x)}2}|D\phi|^2\,dx\,dt.\\
\end{align*}

\end{proof}

\section{The estimate of normalized parabolic $p(x)$-Laplacian}

\end{document}